\pdfoutput=1
\documentclass[preprint]{elsarticle}

\usepackage{mathtools}

\usepackage[colorlinks=true]{hyperref}

\newtheorem{theorem}{Theorem}[section]
\newtheorem{proposition}[theorem]{Proposition}
\newtheorem{lemma}[theorem]{Lemma}
\newtheorem{remark}[theorem]{Remark}
\newproof{proof}{Proof}

\renewcommand{\ge}{\varepsilon}
\newcommand{\R}{\mathbb{R}}
\newcommand{\C}{\mathbb{C}}

\let\colon\relax

\newcommand{\colon}{\nobreak\mskip2mu\mathpunct{}\nonscript   \mkern-\thinmuskip{:}\mskip5muplus1mu\relax}

\makeatletter

\DeclareSymbolFont{AMSb}{U}{msb}{m}{n}

\@ifundefined{mathbb}{%
    \DeclareSymbolFontAlphabet{\mathbb}{AMSb}%
}{}

\DeclareSymbolFont{rsfs}{U}{rsfs}{m}{n}
\DeclareSymbolFontAlphabet{\mathscr}{rsfs}

\DeclareMathAlphabet{\mathfrak}{U}{euf}{m}{n}
\SetMathAlphabet{\mathfrak}{bold}{U}{euf}{b}{n}

\makeatother


\begin{document}


\title{Semiclassical analysis for pseudo-relativistic Hartree equations}
\author[cin]{Silvia Cingolani}
  \ead{silvia.cingolani@poliba.it}
\author[sec]{Simone Secchi}
	\ead{simone.secchi@unimib.it}

\address[cin]{Dipartimento di Meccanica,
  Matematica e Management, Politecnico di Bari, via Orabona 4, I-70125 Bari}

\address[sec]{Dipartimento di Matematica e Applicazioni, Universit\`a di Milano Bicocca, via Roberto Cozzi 55, I-20125 Milano.}
\date{\today}

\begin{abstract}
In this paper we study the semiclassical limit for the pseudo-relativistic Hartree equation
\begin{displaymath}\label{mainequation}
\sqrt{-\ge^2 \Delta + m^2}u + V u = \left( I_\alpha * |u|^{p} \right)
  |u|^{p-2}u, \quad\hbox{in $\R^N$},
\end{displaymath}
where $m>0$,  $2 \leq p < \frac{2N}{N-1}$, $V \colon \R^N \to \R$ is an external scalar
potential,
$
I_\alpha (x) = \frac{c_{N,\alpha}}{|x|^{N-\alpha}}
$
is a convolution kernel, $c_{N,\alpha}$ is a positive constant and
$(N-1)p-N<\alpha <N$.
For $N=3$, $\alpha=p=2$, our equation  becomes the pseudo-relativistic Hartree equation with Coulomb kernel.
\end{abstract}


\maketitle


\section{Introduction}

In this paper we study the semiclassical limit ($\ge \to 0^+$) for the pseudo-relativistic Hartree equation
\begin{equation}\label{eq:1.1}
i \ge \frac{\partial \psi}{\partial t} = \left(\sqrt{- \ge^2 \Delta +m^2} -m \right)\psi  + V \psi  -
\left( \frac{1}{|x|} * |\psi|^2 \right)  \psi, \quad x \in \R^3
\end{equation}
where  $\psi \colon \R \times \R^3 \to \C$ is the wave field, $m>0$ is a physical constant, $\ge$ is the semiclassical parameter $0 < \ge \ll 1$, a dimensionless scaled Planck constant (all other physical constant are rescaled to be $1$), $V$ is bounded external potential in $\R^3$.
Here the pseudo-differential operator $\sqrt{-\ge^2\Delta +m^2}$ is simply defined in Fourier variables by the symbol $\sqrt{\ge^2 |\xi|^2+m^2}$ (see \cite{ll}).

Equation $(\ref{eq:1.1})$ has interesting applications in the quantum theory for large systems of self-interacting, relativistic bosons with mass $m >0$. As recently shown by Elgart and Schlein \cite{es}, equation  $(\ref{eq:1.1})$ emerges as the correct evolution equation for the mean-field dynamics of
many-body quantum systems modelling pseudo-relativistic boson stars in astrophysics. The external potential, $V = V (x)$, accounts for gravitational fields from other stars. In what follows, we will assume that $V$ is a smooth, bounded function (see \cite{ly,fjl,FJL,FJL1,Lenz2,MM}).
The pseudo-relativistic Hartree equation  can be also derived coupling together a pseudo-relativistic Schr\"odinger equation with a Poisson equation (see for instance \cite{AMS,SCV}), i.e.
\[
\left\{
\begin{array}{l}
i \ge \frac{\partial \psi}{\partial t} = \left(\sqrt{- \ge^2 \Delta +m^2} -m \right)\psi +V \psi   - U \psi, \\
-\Delta U = |\psi|^2
\end{array}
\right.
\]
See also \cite{DSS,fl} for recent developments for models involving pseudo-relativistic Bose gases.

Solitary wave solutions $\psi(t,x)= e^{i t \lambda / \ge } u(x)$, $\lambda >0$ to equation $(\ref{eq:1.1})$
lead to solve the non local single equation
 \begin{equation}\label{psedorelativistichartree}
\sqrt{- \ge^2 \Delta +m^2 } u + V u  = \left( \frac{1}{|x|} * |u|^2 \right) u, \quad\hbox{in $\R^3$}
\end{equation}
where for simplicity we write $V$ instead of $V + (\lambda -m)$.

More generally, in this paper we will study the generalized pseudo-relativistic Hartree equation
\begin{equation} \label{eq:1generalizzato}
\sqrt{-\ge^2 \Delta + m^2}u + Vu = \left( I_\alpha * |u|^{p} \right)
  |u|^{p-2}u, \quad\hbox{in $\R^N$},
\end{equation}
where $m>0$,  $2 \leq p <
\frac{2N}{N-1}$, $V \colon \R^N \to \R$ is an external scalar
potential,
\[
I_\alpha (x) = \frac{c_{N,\alpha}}{|x|^{N-\alpha}} \quad (x \neq 0), \quad \alpha \in (0,N)
\]
is a convolution kernel and $c_{N,\alpha}$ is a positive constant; for
our purposes we can choose $c_{N,\alpha}=1$.
For $N=3$, $\alpha=p=2$, equation \textcolor{blue}{$(\ref{eq:1generalizzato})$}
becomes the pseudo-relativistic Hartree equation $(\ref{psedorelativistichartree})$ with Coulomb kernel.

We refer to \cite{ww,CSS,CCS1,MVS} for the semiclassical analysis of the non-relativistic Hartree equation.
The study of the pseudo-relativistic Hartree  equation $(\ref{psedorelativistichartree})$ without external potential $V$
starts in the pioneering paper \cite{ly} where Lieb and Yau,
by minimization on the sphere $\{\phi \in L^2(\R^3) \mid \int_{\mathbb{R}^3} |\phi|^2 =M \}$, proved that a radially symmetric ground state exists in $H^{1/2}(\R^3)$ whenever $M<M_c$, the so-called Chandrasekhar mass. Later Lenzmann proved in \cite{Lenz} that this ground state is unique (up to translations and phase change) provided that the mass $M$ is sufficiently small; some results about the non-degeneracy of the ground state solution are also given.

Successively, in \cite{CZN1} Coti-Zelati and Nolasco proved existence of a positive radially symmetric ground state solution for a pseudo-relativistic Hartree equation without external potential $V$, involving a more general radially symmetric convolution kernel. See the recent paper \cite{CZN2} dealing existence of ground states with given fixed ``mass-charge".

In \cite{MeZo} Melgaard and Zongo established that $(\ref{psedorelativistichartree})$  has a sequence of radially symmetric solutions of higher and higher energy, assuming that $V$ is radially symmetric potential.

The requirement that $V$ has radial symmetry was dropped in the recent paper \cite{CS14}, where a positive ground state solution for
the pseudo-relativistic Hartree equation    \textcolor{blue}{$(\ref{eq:1generalizzato})$} is constructed under the assumption $(N-1) p -N < \alpha < N$.

To the best of our knowledge the study of the semiclassical limit for the pseudo-relativistic Hartree equation has been considered by Aki, Markowich and Sparber in \cite{AMS}. Using Wigner trasformation techniques, they showed that its semiclassical limit
yields the well known relativistic Vlasov-Poisson system.

In the present paper we are interested to study the pseudo-relativistic Hartree equation in the semiclassical limit regime $(0< \ge \ll 1)$, using variational methods.
Replacing $u (y)$ by $\ge^{\frac{\alpha}{2(1-p)}} u(\ge y)$, equation (\ref{eq:1generalizzato}) becomes equivalent to following Hartree equation	
\begin{equation} \label{eq:1bis}
\sqrt{- \Delta + m^2}  u   + V_\ge(y)  u = \left( I_\alpha * |u|^{p} \right)
	| u|^{p-2}  u, \quad\hbox{in $\R^N$}.
\end{equation}
where $V_\ge (y)=V(\ge y)$.
In what follows we will assume that
\begin{description}
\item[(V)] $V \colon \R^N \to \R$ is a continuous and bounded function
  such that $V_{\mathrm{min}} = \inf_{\R^N} V >-m$ and there
  exists a bounded open set $O \subset R^N$ with the property that
\[
V_0 = \inf_{O} V < \min_{\partial O} V.
\]
\end{description}
Let us define
\[
\mathscr{M}= \left\{ y \in O \mid V(y)= V_0 \right\}.
\]

We will establish the existence of a single-spike solution concentrating around a point close to
$\mathscr{M}$. Precisely, our main result is the following.

\begin{theorem} \label{th:main}
Retain assumption (\textbf{V}) and
  assume that $2 \leq p < 2N/(N-1)$ and $(N-1)p-N<\alpha <N$. Then, for every sufficiently small
  $\ge>0$, there exists a solution $u_\ge \in H^{1/2}(\R^N)$
  of equation $(\ref{eq:1bis})$ such that $u_\ge$ has a local maximum point
  $y_\ge$ satisfying
\[
\lim_{\ge \to 0} \mathrm{dist}(\ge y_\ge,\mathscr{M})=0,
\]
and for which
\[
u_\ge(y) \leq C_1 \exp \left( -C_2 |y-  y_\ge| \right)
\]
for suitable constants $C_1>0$ and $C_2>0$. Moreover, for any sequence
$\{\ge_n\}_n$ with $\ge_n \to 0$, there exists a
subsequence, still denoted by the same symbol, such that there exist a
point $y_0 \in \mathscr{M}$ with $\ge_n  y_{\ge_n} \to y_0$, and a positive least-energy solution $U \in H^{1/2}(\mathbb{R}^N)$ of
the equation
\[
\sqrt{-\Delta + m^2} \ U + V_0 U = \left( I_\alpha * U^{p} \right) U^{p-1}
\]
for which we have
\begin{equation} \label{eq:26}
u_{\ge_n}(y) = U \left(y- y_{\ge_n}\right) + \mathscr{R}_n(y)
\end{equation}
where $\lim_{n \to +\infty} \|\mathscr{R}_n\|_{H^{1/2}} =0$.
\end{theorem}

\bigskip
To prove the main result,
 we replace the \emph{nonlocal}
problem (\ref{eq:1generalizzato}) in $\R^N$ with a local Neumann problem in the
half space $\R_{+}^{N+1}$ as in \cite{CZN1} (see \cite{Caff}).
We will find critical points of the Euler functional associated to the local Neumann problem by means of a variational approach introduced in \cite{byeonjeanjean, byeonjeanjean1} (see also \cite{CJS}) for nonlinear Schr\"odinger equations and extended in \cite{CSS} to deal with non-relativistic Hartree equations.

In the present paper the presence of a pseudo-differential operator combined with a nonlocal term requires new ideas. As a first step, we need to perform a deep analysis of the local realization of the following limiting problem
\begin{equation} \label{eq:5zz}
\sqrt{-\Delta + m^2} u + a u = \left( I_\alpha * |u|^p \right)|u|^{p-2} u
\end{equation}
with~$a>-m$. This equation does not have a unique (up to translation) positive, ground state solution, apart from the case $p=2$, $N=3$.
Nevertheless we can prove that the set of positive, ground state solutions to the local realization of equation $(\ref{eq:5zz})$
satisfies  some compactness properties. This is the crucial tool for finding single-peak solutions which are close to a set of prescribed functions. Even if we use a purely variational approach, we will take into account the shape  and  the location
of the expected solutions as in the reduction methods.

Recently the existence of a spike-pattern solution for fractional nonlinear Schr\"odinger equation has been  proved by
Davila, del Pino and Wei in the semiclassical limit regime (see \cite{daviladelpinowei}).
The authors perform a refined Lyapunov-Schmidt reduction, taking into advantage the fact that
the limiting fractional problem has an unique,  positive, radial, ground state solution, which is nondegenerate.

\begin{center}
\textbf{Notation}
\end{center}
\begin{itemize}
\item We will use $|\cdot|_q$ for the norm in $L^q$, and
  $\|\cdot\|$ for the norm in $H^1(\mathbb{R}_{+}^{N+1})$.
\item Generic positive constants will be denoted by the (same) letter
  $C$.
\item The symbol $\R_{+}^{N+1}$ denotes the half-space $\{ (x,y) \mid
  x>0, \ y \in \R^N\}$. We will identify the boundary~$\partial
  \R_{+}^{N+1}$ with $\R^N$.
\item The symbol $*$ will denote the convolution of two functions.
\item For any subset $A$ of $\R^N$ and any
  $\varrho>0$, we set $A^\varrho = \left\{ y \mid \mbox{dist}(y,A) \leq \varrho \right\}$.

\item For any subset $A$ of $\R^N$ and any
  $\varrho>0$, we set $A_\varrho = \left\{ y \mid \varrho y \in A \right\}$.
\end{itemize}

\section{Preliminaries and variational setting}

The realization of the operator $\sqrt{m^2-\ge^2\Delta}$ in Fourier
variables seems not convenient for our purposes. Therefore, we prefer to
make use of a \emph{local} realization (see \cite{CZN1,Caff}) by
means of the \emph{Dirichlet-to-Neumann} operator defined as
follows.

For any $\ge >0$,  given $u \in \mathcal{S}(\R^N)$,
the Schwartz space of rapidly decaying smooth functions defined on
$\R^N$, there exists one and only one function $v \in
\mathcal{S}(\R^{N+1}_{+})$ such that
{\begin{displaymath} 
\left\{
\begin{array}{ll}
- \ge^2 \Delta v + m^2 v =0 &\hbox{in $\R^{N+1}_{+}$}\\
v(0,y)=u(y) &\hbox{for $y \in \R^N = \partial
\R^{N+1}_{+}$.}
\end{array}
\right.
\end{displaymath}
Setting
\[
T_\ge u(y)=- \ge \frac{\partial v}{\partial x}(0,y),
\]
we easily see that the problem
{\begin{displaymath} 
\left\{
\begin{array}{ll}
- \ge^2 \Delta w + m^2 w =0 &\hbox{in $\R^{N+1}_{+}$}\\
w(0,y)=T_\ge u(y) &\hbox{for $y \in \partial \R^{N+1}_{+} = \{ 0\}
\times \R^N \simeq \R^N$}
\end{array}
\right.
\end{displaymath}
is solved by $w(x,y)=- \ge \frac{\partial v}{\partial
x}(x,y)$. From this we deduce that
\[
T_\ge (T_\ge u)(y)=- \ge \frac{\partial w}{\partial x}(0,y) = \ge^2 \frac{\partial^2
v}{\partial x^2}(0,y) = \left(- \ge^2 \Delta_y v + m^2 v \right)(0,y),
\]
and hence $T_\ge \circ T_\ge = (- \ge^2 \Delta_y +m^2)$, namely $T_\ge$ is a square
root of the Schr\"{o}dinger operator $-\ge^2 \Delta_y + m^2 $ on
$\R^{N}=\partial \R_{+}^{N+1}$.

\vspace{10pt}

From the previous construction, we can replace the \emph{nonlocal}
problem $(\ref{eq:1generalizzato})$
 in $\R^N$ with the local Neumann problem in the
half space $\R_{+}^{N+1}$
{\begin{displaymath} 
\left\{
\begin{array}{ll}
- \ge^2 \Delta v (x,y) + m^2 v(x,y)  =0 &\hbox{in $\R_{+}^{N+1}$} \\
-\ge \frac{\partial v}{\partial x}(0,y) = -V(y)v (0,y) +
\left( I_\alpha * |v(0,\cdot)|^{p} \right) |v(0,y)|^{p-2}v(0,y)
&\hbox{for $y \in \R^N$}.
\end{array}
\right.
\end{displaymath}
}
Setting  $v_\ge(x,y)=\ge^{\frac{\alpha}{2(1-p)}} 
v(\ge x, \ge y)$ and 
$V_\ge (y)=V(\ge y)$,
 we are led to the \emph{local}
boundary-value problem
\begin{displaymath}
\left\{ \begin{array}{ll}
  -\Delta v_\ge + m^2 v_\ge = 0 &\hbox{in $\R_{+}^{N+1}$}\\
  -\frac{\partial v_\ge}{\partial x}(0,y)= -V_\ge(y)v_\ge(0,y) + \left( I_\alpha * |v_\ge(0,\cdot)|^{p} \right) |v_\ge(0,y)|^{p-2}v_{\ge}(0,y) &\hbox{for $y \in \R^N$}.
\end{array}
\right.
\end{displaymath}
We introduce the Sobolev space $H=H^1(\R_{+}^{N+1})$, and recall that
there is a continuous \emph{trace operator} $\gamma \colon H \to
H^{1/2}(\R^N)$. Moreover, this operator is surjective and the
inequality
\[
|\gamma (v)|_p^p \leq p |v|_{2(p-1)}^{p-1} \left| \frac{\partial v}{\partial x} \right|_2
\]
holds for every $v \in H^1(\R_{+}^{N+1})$: we refer to \cite{Tartar} for basic facts about the Sobolev space $H^{1/2}(\mathbb{R}^N)$ and the properties of the trace operator.

Reasoning as in \cite[Page 5]{CS14} and taking the
Hardy-Littlewood-Sobolev inequality (see \cite[Theorem 4.3]{LiebLoss}) into consideration,
it follows easily that the functional $\mathscr{E}_\ge
\colon H \to \R$ defined by
{\setlength\arraycolsep{2pt}
\begin{eqnarray*} 
\mathscr{E}_\ge (v) &=& \frac{1}{2} \int_{\R^{N+1}_{+}} |\nabla v|^2 \, dx\, dy  +\frac{m^2}{2} \int_{\R^{N+1}_{+}} v^2 \, dx\, dy \nonumber\\
&&{}+ \frac{1}{2} \int_{\R^N} V_\ge(x) \gamma(v)^2 \, dy -\frac{1}{2p } \int_{\R^N} \left( I_\alpha * |\gamma (v)|^p \right) |\gamma (v)|^p \, dy
\end{eqnarray*}
}
is of class $C^1$, and its critical points are (weak) solutions
to problem \textcolor{blue}{$(\ref{eq:1bis})$.}

\section{Compactness properties for the limiting problem}

For $a>-m$, the equation
\begin{equation} \label{eq:5}
\sqrt{-\Delta + m^2} u + a u = \left( I_\alpha * |u|^p \right)|u|^{p-2} u
\end{equation}
plays the r\^ole of a limiting problem for (\ref{eq:1bis}). Its Euler
functional $L_a \colon H \to \R$ is defined (via the local realization of Section 2) by
\begin{multline*}
  L_a (v) = \frac{1}{2} \int_{\R_{+}^{N+1}} \left( |\nabla v|^2 + m^2 |v|^2 \right) \, dx \, dy \\
  {}+ \frac{a}{2} \int_{\R^N} |\gamma (v)|^2 \, dy - \frac{1}{2p} \int_{\R^N} \left( I_\alpha * |\gamma (v)|^p \right) |\gamma (v)|^p \, dy.
\end{multline*}
We define the ground-state level
\begin{displaymath}
m_a = \inf \left\{ L_a (v) \mid L'_a(v)=0, v \in H \setminus \{0\} \right\}
\end{displaymath}
and the set $S_a$ of elements $v \in H \setminus \{0\}$ such that
$v > 0$,
$L_a(v)=m_a$, and for every $x \geq 0$:
\begin{equation} \label{eq:24}
\max_{y \in \R^N} v(x,y) = v(x,0).
\end{equation}
\begin{proposition}
The set $S_a$ is non-empty for any $a>-m$.
\end{proposition}
\begin{proof}
The proof is indeed standard, and we will be sketchy. 
First of all, we invoke \cite[Lemma 2.1]{CZN2} to deduce that ground states of $L_a$ correspond to ground states of the functional $\mathcal{L}_a \colon H^{1/2}(\mathbb{R}^N) \to \mathbb{R}$ defined as
\begin{multline} \label{eq:38}
\mathcal{L}_a(u) =
\frac{1}{2}\int_{\mathbb{R}^N} \left(\left| \sqrt{\left( m^2-\Delta \right)^{1/2}-m}u \right|^2 + (a+m)|u|^2 \right) \\
{}-\frac{1}{2p} \int_{\mathbb{R}^N} \left( I_\alpha * |u|^p \right) |u|^p.
\end{multline}
We claim that $\mathcal{L}_a$ possesses a ground state.
We fix $a>-m$ and consider the minimization problem associated to (\ref{eq:38})
\begin{equation} \label{eq:37}
\widetilde{m}_a = 
\inf_{u\in H^{1/2}(\mathbb{R}^N) \setminus \{0\}} 
\frac{\int_{\mathbb{R}^N} |\sqrt{ (m^2-\Delta)^{1/2}-m } u|^2+ \left( a+m \right) |u|^2}{\left( \int_{\mathbb{R}^N} \left( I_\alpha * |u|^p \right) |u|^p \right)^{\frac{1}{p}}}.
\end{equation}
Since \(\sqrt{m^2-\Delta}-m>0\)
in the sense of functional calculus and $a+m>0$, it follows easily that $\widetilde{m}_a>0$. As in \cite[Proof of Proposition 2.2]{MVS13} we can show that $\widetilde{m}_a$ is attained.
Since the quotient in (\ref{eq:37}) is homogeneous of degree zero, as in the local case we see that any minimizer of
$\widetilde{m}_a$ is, up to a rescaling and a translation, a ground state for (\ref{eq:38}). Therefore the claim is proved, and in particular $S_a \neq \emptyset$. It is easy to check that ground states are non-negative, and, as in \cite[Theorem 5.1]{CZN1}, actually strictly positive. 
\end{proof}
\begin{remark}
By \cite[Formula (A.3)]{ly}, the quotient to be 
minimized in (\ref{eq:37}) decreases under polarization. 
This implies, reasoning as in \cite[Section 5]{MVS13} (see also \cite{DaSS}) that 
ground states are radially symmetric around a point of $\mathbb{R}^N$.
\end{remark}
For $U \in S_a$, we write $E_a = L_a(U)$. By an immediate extension of
\cite[Lemma 3.17]{R}, the map $a \mapsto E_a$ is
strictly increasing and continuous.  The following is the main result of this
section.
\begin{proposition} \label{prop:comp}
The set $S_a$ is compact in $H$,
  and for some~$C>0$ and any~$\sigma \in (-V_{\mathrm{min}},m) \cap [0,+\infty)$ we have
\begin{equation} \label{eq:27}
v(x,y) \leq C e^{-(m-\sigma)\sqrt{x^2+|y|^2}} e^{-\sigma x}
\end{equation}
for every $v \in S_a$.
\end{proposition}
\begin{proof}
  If $v \in S_a$, it follows easily from \cite[Theorem 5.1]{CZN1} or
  \cite[Theorem 7.1]{CS14} that $v$ decays exponentially fast at
  infinity and (\ref{eq:27}) holds. Moreover, since
\[
m_a = L_a (v) = \left( \frac{1}{2} - \frac{1}{2p} \right) \left( |\nabla v|_2^2 + m^2 |v|_2^2 \right),
\]
$S_a$ is bounded in $H$. We claim that $S_a$ is also bounded in
$L^\infty(\R_{+}^{N+1})$.

Indeed, 
by \cite[Theorem 3.2]{CZN1} it follows that $\gamma(v) \in L^q (\R^N)$ for any $q \in [2, \infty]$,  
then also 
$g(\cdot) = -a \gamma (v) + \left( I_\alpha * |\gamma(v)|^p \right) |\gamma (v)|^{p-2}\gamma(v) \in L^q(\R^N)$
for $q \in [2, \infty]$. 
Following \cite{CSM}, we let  $u(x,y)=\int_0^x v(t,y)\, dt$. It follows that $u \in H^1((0,R) \times \R^N)$ for all $R>0$.
Arguing as in \cite[Proposition 3.9]{CZN1},  we can deduce that $u$ is a weak solution of the Dirichlet problem
\begin{equation} \label{eq:31}
\left\{ \begin{array}{ll}
-\Delta u + m^2 u = g &\hbox{in $\R_{+}^{N+1}$}\\
u=0 &\hbox{for $y \in \R^N$}.
\end{array}
\right.
\end{equation}
where $g(x,y)= g(y)$ for every $x >0$ and $y \in \R^N$. We sketch the proof for the sake of completeness. Pick an arbitrary function $\eta \in C_0^\infty(\mathbb{R}_{+}^{N+1})$ and write $\omega_t(x,y)=\eta(x+t,y)$ for any $t \geq 0$. Then
\begin{multline*}
\int_0^{+\infty} \int_0^{+\infty} \int_{\mathbb{R}^N} \nabla v(x,y) \cdot \nabla \eta (x+t,y) \, dy \, dx\, dt \\
=\int_0^{+\infty} \int_x^{+\infty} \int_{\mathbb{R}^N} \nabla v(x,y) \cdot \nabla \eta (s,y) \, dy \, ds \, dx 
\\
= \int_0^{+\infty} \int_0^s \int_{\mathbb{R}^N} \nabla v(x,y) \cdot \nabla \eta (s,y) \, dy \, dx \, ds \\
=\int_0^{+\infty} \int_{\mathbb{R}^N} \nabla \left( \int_0^s v(x,y)\, dx \right) \cdot \nabla \eta (s,y) \, dy \, ds
\end{multline*}
and this readily implies that 
\begin{equation*}
\int_{\mathbb{R}_{+}^{N+1}} \left( \nabla v \cdot \nabla w_t + m^2 v w_t \right) \, dx \, dy = \int_{\mathbb{R}^N} g w_t \, dy.
\end{equation*}
An integration with respect to $t$ from $0$ to $+\infty$ gives
\[
\int_{\mathbb{R}^{N+1}_{+}} \left( \nabla u \cdot \nabla \eta + m^2 u\eta - g \eta \right) \, dx \, dy = 0,
\]
and hence the validity of (\ref{eq:31}) is proved.

Moreover for  any given $R>0$ we can define
$u_{\mathrm{odd}} \in H^1((-R,R) \times \mathbb{R}^N)$ and
$g_{\mathrm{odd}} \in \bigcap_{q \geq 2} L^q((-R,R) \times
\mathbb{R}^N)$ by
\begin{displaymath}
u_{\mathrm{odd}} = \left\{
\begin{array}{ll}
 u(x,y) &\hbox{if $x \geq 0$} \\
 -u(-x,y) &\hbox{if $x <0$},
\end{array}
\right.
\qquad
g_{\mathrm{odd}}(x,y)= \left\{
\begin{array}{ll}
 g(y) &\hbox{if $x \geq 0$}\\
 -g(y) &\hbox{if $x<0$}.
\end{array}
\right.
\end{displaymath}
It is easy to check as before that
\[
-\Delta u_{\mathrm{odd}} + m^2 u_{\mathrm{odd}} = g_{\mathrm{odd}}
\quad\hbox{in $\mathbb{R}^{N+1}$}.
\]
Since $g_{\mathrm{odd}} \in L^q((-R,R)\times \R^N)$ for any $q \in [2, +\infty[$, $R>0$, we can invoke standard regularity results to conclude
that
\[
u_{\mathrm{odd}} \in W^{2,q} ((-R,R) \times \mathbb{R}^{N})
\]
for
every $q \geq 2$ and every $R>0$, and hence $u_{\mathrm{odd}} \in
C^{1,\beta}(\mathbb{R}^{N+1})$, $u \in C^{1,\beta}(\mathbb{R}_{+}^{N+1})$ and $v = \frac{\partial u}{\partial x} \in
C^{0,\beta}(\mathbb{R}_{+}^{N+1})$ by Sobolev's Embedding
Theorem. Therefore $g \in C^{0,\beta/(p-1)}(\mathbb{R}^N)$, and
Schauder estimates yield $u \in C^{2,\beta/(p-1)}(\mathbb{R}_{+}^{N+1})$ and $v \in C^{1,\beta/(p-1)}(\mathbb{R}_{+}^{N+1})$.  Moreover, the $C^{1,\beta}$-norm of $v$ can be
estimated by the $L^q$-norm of $g$, which immediately implies that
$S_a$ is a bounded subset of $L^\infty(\mathbb{R}_{+}^{N+1})$.




Next, we claim that $\lim_{|(x,y)| \to +\infty} v(x,y)=0$
uniformly with respect to $v \in S_a$. We assume by contradiction that
this is false: there exist a number $\delta >0$, a sequence of points
$(x_n,y_n) \in \R_{+}^{N+1}$ and a sequence of elements $v_n \in S_a$
such that $x_n+|y_n| \to +\infty$ but $v_n(x_n,y_n) \geq \delta$ for
every $n$. Let us write $z_n=(x_n,y_n)$, and call
$\tilde{v}_n(z)=v_n(z+z_n)$ for $z=(x,y) \in \R_{+}^{N+1}$. By the
previous arguments, $\{\tilde{v}_n\}_n$ is a bounded sequence in $H
\cap L^\infty(\R_{+}^{N+1})$. Moreover, up to a subsequence, we can
assume that $v_n \rightharpoonup v$, $\tilde{v}_n \rightharpoonup
\tilde{v}$ in $H$ and locally uniformly in $\R_{+}^{N+1}$. As
in \cite[pag. 989]{CSS}, both $v$ and $\tilde{v}$
weakly solve (\ref{eq:5}). We now show that they are non-trivial weak
solutions. The conclusion is obvious for $\tilde{v}$, since
$\tilde{v}_n(0) = v_n(z_n) \geq \delta$, so that $\tilde{v}(0) \geq
\delta$. We consider instead $v$, and remark that \cite[Eq. (3.16)]{CZN1} implies
\[
\sup_{y \in \R^{N}} |v_n(x,y)| \leq C |\gamma(v_n)|_2 \ e^{-mx}
\]
for some universal constant $C>0$. Hence $\delta \leq v_n(z_n) \leq
|\gamma(v_n)|_2 \ e^{-mx_n}$, and the boundedness of $\gamma(v_n)$ in
$L^2$ yields the boundedness of $\{x_n\}_n$ in $\R$. Without loss of
generality, we can assume that $x_n \to \bar{x} \in [0,+\infty)$.
Therefore, by (\ref{eq:24}),
\[
v_n(\bar{x},0) \geq v_n(\bar{x},y_n) \geq v_n(x_n,y_n)+o(1) \geq
\frac{\delta}{2}
\]
by locally uniform convergence, and we conclude that $v$ is also
nontrivial.

Now, for every $n \in \mathbb{N}$,
\[
L_a(v_n) = \left(\frac{1}{2} - \frac{1}{2p} \right) \left(
  \int_{\R_{+}^{N+1}} \left( |\nabla v_n|^2 + m^2 v_n^2 \right)\, dx\,
  dy + a \int_{\R^N} \gamma(v_n)^2 \, dy \right) =m_a,
\]
and
\[
L_a(v) \geq m_a, \quad L_a (\tilde{v})\geq m_a.
\]
If $R>0$ satisfies $2R \leq x_n + |y_n|$, then
{\setlength\arraycolsep{2pt}
\begin{eqnarray*}
  m_a &=& L_a(v_n) \\
  &\geq& \left( \frac{1}{2} - \frac{1}{2p} \right) \liminf_{n \to +\infty} \int_{B(0,R)} \left( |\nabla v_n|^2 + m^2 v_n^2 \right) \, dx \, dy \\
  &&+ a \int_{B(0,R)\cap \left( \{0\} \times \R^N \right)} \gamma(v_n)^2\, dy \\
  &&+ \left( \frac{1}{2} - \frac{1}{2p} \right) \liminf_{n \to +\infty} \int_{B(0,R)} \left( |\nabla \tilde{v}_n|^2 + m^2 \tilde{v}_n^2 \right) \, dx \, dy \\
  &&+ a \int_{B(0,R)\cap \left( \{0\} \times \R^N \right)} \gamma(\tilde{v}_n)^2 \, dy \\
  &\geq&  \left( \frac{1}{2} - \frac{1}{2p} \right) \left( \int_{B(0,R)} \left( |\nabla \tilde{v}|^2 + m^2 \tilde{v}^2 \right) \, dx \, dy + a \int_{B(0,R)\cap \left( \{0\} \times \R^N \right)} \gamma(\tilde{v})^2 \, dy
  \right) \\
  &=& L_a(v)+L_a(\tilde{v}) + o(1) = 2m_a + o(1)
\end{eqnarray*}
}
as $R \to +\infty$. This contradiction proves that 
\begin{equation} \label{eq:30}
\lim_{|(x,y)| \to +\infty} v(x,y)=0 \quad\text{uniformly with respect to $v \in S_a$}.
\end{equation}
From \cite[page 70]{CZN1} it follows immediately that
\begin{displaymath}
  \lim_{|y| \to +\infty} I_\alpha *|\gamma (v)|^p(y)=0, \quad\mbox{uniformly w.r.t. $v \in S_a$}.
\end{displaymath}
Pick $R_a>0$, independent of $v \in S_a$, such that $|y| \geq R_a$ implies
\[
\left| I_\alpha * |\gamma(v)|^p(y) \right| \left|\gamma(v) (y)
\right|^{p-2} \leq \frac{a}{2}.
\]
As a consequence,
\[
\left\{
\begin{array}{ll}
  -\Delta v + m^2 v = 0 &\mbox{in $\R_{+}^{N+1}$} \\
  -\frac{\partial v}{\partial x} \leq -\frac{a}{2} v &\mbox{in $\{0\} \times \{|y| \geq R_a \}$}
\end{array}
\right.
\]
As in \cite[Theorem 5.1]{CZN1} or \cite[Theorem 7.1]{CS14}, and recalling the uniform decay at infinity of (\ref{eq:30}), it
follows that $v$ decays exponentially fast at infinity, with constants
that are uniform with respect to $v \in S_a$.

We are ready to conclude: let $\{v_n\}_n$ be a sequence from
$S_a$. Our previous arguments show that \(\{v_n\}_n\) converges --- up
to a subsequence --- weakly to some $v \in H$, and this limit $v$ is
also a solution to equation~(\ref{eq:5}).
Fix
\[
r > \max \left\{ 1, \frac{N}{N(2-p)+p} \right\}
\]
and split $I_\alpha$ as $I_\alpha^1+I_\alpha^2$, where $I_\alpha^1 \in
L^r(\R^N)$ and $I_\alpha^2 \in L^\infty (\R^N)$. This induces a
decomposition of the non-local term
$\mathscr{N}(v)=\mathscr{N}^1(v)+\mathscr{N}^2(v)$ as
{\setlength\arraycolsep{2pt}
\begin{eqnarray*}
\mathscr{N}(v) &=& \frac{1}{2p} \int_{\R^{N}} \left( I_\alpha * |\gamma (v)|^p \right) |\gamma (v)|^p\, dy \\
\mathscr{N}^1(v) &=& \frac{1}{2p} \int_{\R^{N}} \left( I_\alpha^1 * |\gamma (v)|^p \right) |\gamma (v)|^p\, dy \\
\mathscr{N}^2(v) &=& \frac{1}{2p} \int_{\R^{N}} \left( I_\alpha^2 * |\gamma (v)|^p \right) |\gamma (v)|^p\, dy.
\end{eqnarray*}
}
We obtain immediately that
{\setlength\arraycolsep{2pt}
\begin{eqnarray}
0&=&\lim_{n \to +\infty} \left( \int_{\R_{+}^{N+1}} \left( |\nabla v_n|^2 + m^2 v_n^2 \right) \, dx\, dy - \mathscr{N}(v_n) \right) \nonumber \\
&=&\int_{\R_{+}^{N+1}} \left( |\nabla v|^2 + m^2 v^2 \right) \, dx\, dy - \mathscr{N}(v). \label{eq:6}
\end{eqnarray}
} We complete the proof by showing that $\lim_{n \to +\infty}
\mathscr{N}(v_n) = \mathscr{N}(v)$. Now, by the
Hardy-Littlewood-Sobolev inequality (see \cite[Theorem 4.3]{LiebLoss})
{\setlength\arraycolsep{2pt}
\begin{eqnarray*}
&&\left\vert
\mathscr{N}^1(v_{n})-\mathscr{N}^1(v)\right\vert \\
&\leq&
\int_{\mathbb{R}^{N}\times \mathbb{R}^N} I_\alpha^1 (x-y)\left\vert
|\gamma(v_{n}) (x)|^{p}|
\gamma(v_{n})(y)|^{p}-|\gamma(v)(x)|^{p}|\gamma(v)(y)|^{p}\right\vert dx dy \nonumber \\
&=& \int_{\mathbb{R}^{N}\times \mathbb{R}^N} I_\alpha^1 (x-y) \Big| |\gamma(v_{n}) (x)|^{p}|
\gamma(v_{n})(y)|^{p}-|\gamma(v_n)(x)|^p |\gamma(v)(y)|^p \nonumber \\
&&{} + |\gamma(v_n)(x)|^p |\gamma(v)(y)|^p -|\gamma(v)(x)|^{p}|\gamma(v)(y)|^{p}  \Big| dx\, dy \nonumber\\
&\leq& \int_{\mathbb{R}^{N}\times \mathbb{R}^N} I_\alpha^1 (x-y) |\gamma(v_n)(x)|^p \left| |\gamma(v_n)(y)|^p - |\gamma(v)(y)|^p \right|\, dx\, dy \nonumber \\
&&{} + \int_{\mathbb{R}^{N}\times \mathbb{R}^N} I_\alpha^1 (x-y) |\gamma(v)(y)|^p \left| |\gamma(v_n)(x)|^p - |\gamma(v)(x)|^p \right| \, dx\, dy
\nonumber \\
&=& 2
\int_{\mathbb{R}^N \times \mathbb{R}^{N}}
I_\alpha^1 (x-y) |\gamma(v_{n})(x)|^{p} \left\vert
|\gamma(v_{n})(y)|^{p}-|\gamma(v)(y)|^{p}\right\vert
dxdy\nonumber \\
 &\leq& 2C |I_\alpha^1|_r \left|\gamma(v_n)\right|^{p}_{\frac{2r p}{2r-1}} \left| |\gamma(v_n)|^p - |\gamma(v)|^p
\right|_{\frac{2r}{2r-1}} = o(1),
\end{eqnarray*}
} since $|\gamma(v_n)|^p - |\gamma(v)|^p \to 0$ strongly in
$L_{\mathrm{loc}}^{\frac{2r}{2r-1}}(\mathbb{R}^N)$ by the choice of
$r$. On the other hand,
\begin{multline*}
\left| \mathscr{N}^2(v_n)-\mathscr{N}^2(v) \right| \\
\leq \|I_\alpha^2\|_\infty \int_{\R^N \times \R^N} \left\vert
|\gamma(v_{n}) (x)|^{p}|
\gamma(v_{n})(y)|^{p}-|\gamma(v)(x)|^{p}|\gamma(v)(y)|^{p}\right\vert dx \, dy
\end{multline*}
and the conclusion follows as before. Since $\lim_{n \to +\infty}
\mathscr{N}(v_n) = \mathscr{N}(v)$, equation~(\ref{eq:6}) yields
$\lim_{n \to +\infty} \|v_n\|^2=\|v\|^2$, and the proof is complete.
\end{proof}

\section{The penalization scheme}

For
\[
\delta = \frac{1}{10}\, \mathrm{dist}(\mathscr{M},\R^N \setminus O ) \quad\hbox{and}\quad
\beta \in (0,\delta)
\]
we fix a cut-off $\varphi \in
C_0^\infty(\R_{+}^{N+1})$ such that $0 \leq \varphi \leq 1$
everywhere, $\varphi(x,y)=1$ if $x+|y| \leq \beta$, and
$\varphi(x,y)=0$ if $x+|y| \geq 2\beta$. Setting
$\varphi_\ge(x,y)=\varphi (\ge x, \ge y)$, for any $U \in S_{V_0}$
and any point $y_0 \in \mathscr{M}^\beta$
we define
\[
U_\ge^{y_0} (x,y) = \varphi_\ge \left( x , y -\frac{y_0}{\ge} \right) U\left( x , y -\frac{y_0}{\ge} \right)
\]
We also define, for all $\ge>0$,
\[
\chi_\ge (y) = \left\{
\begin{array}{ll}
0 &\mbox{if $y \in O_\ge$} \\
\ge^{-6/\mu} &\mbox{if $y \notin O_\ge$}
\end{array}
\right.
\]
and
\[
Q_\ge(v) = \left( \int_{\R^N} \chi_\ge \gamma(v)^2 \, dy -1 \right)_{+}^{\frac{2p+1}{2}}
\]
for $v \in H$.
Finally, let
\[
\Gamma_\ge (v) = \mathscr{E}_\ge (v) + Q_\ge (v), \quad v \in H.
\]
We want to find a solution, for $\ge>0$ sufficiently small, near the
set
\[
X_\ge = \left\{ U_\ge^{y_0} \mid \mbox{$y_0 \in \mathscr{M}^\beta$ and
    $U \in S_{V_0}$} \right\}.
\]
We define the (trivial) path $\psi_\ge (s) = s U_\ge^{y_0}$ for every
$s \in [0,1]$.
\begin{lemma} \label{lem:1} There exists $T>0$ such that
  $\Gamma_\ge(\psi_\ge(T))<-2$ for all $\ge$ sufficiently
  small. Moreover,
  \[
    \lim_{\ge \to 0} \max_{s \in [0,T]} \Gamma_\ge (\psi_\ge(s))=E_{V_0}
  \]
where we recall that $E_{V_0}=L_{V_0}(U)$ for $U \in
S_{V_0}$.
\end{lemma}
\begin{proof}
  Indeed, by our definition of the penalization term $Q_\ge$, by a simple change of variables and by the
  exponential decay of $U$ at infinity,
{\setlength\arraycolsep{2pt}
\begin{eqnarray*}
\Gamma_\ge(\psi_\ge(s)) &=& \mathscr{E}_\ge (\psi_\ge(s)) \\
&=& \frac{s^2}{2} \int_{\R_{+}^{N+1}} |\nabla \psi_\ge (s)|^2 + \frac{m^2 s^2}{2} \int_{\R_{+}^{N+1}} \psi_\ge(s)^2 + \frac{s^2}{2} \int_{\R^N} V_\ge \gamma(\psi_\ge(s))^2 \\
&&- \frac{s^{2p}}{2p} \int_{\R^N} \left( I_\alpha * |\gamma(\psi_\ge(s))|^p \right) |\gamma (\psi_\ge(s))|^p \\
&=& \left(\frac{1}{2} \int_{\R_{+}^{N+1}} |\nabla U|^2 + \frac{m^2}{2} \int_{\R_{+}^{N+1}} U^2 + \frac{1}{2} \int_{\R^N} V_0 \gamma(U)^2 +o(1) \right)s^2 \\
&&- \left( \int_{\R^N} \left( I_\alpha * |U|^p \right) |U|^p+o(1) \right) \frac{s^{2p}}{2p}
\end{eqnarray*}
}
where $o(1) \to 0$ as $\ge \to 0$ uniformly with respect to $s$. The conclusion follows easily.
\end{proof}
We are ready to introduce our mini-max scheme. For $\ge>0$ sufficiently
small, we define the set of paths
\[
\Phi_\ge = \left\{ \psi \in C([0,T],H) \mid \psi(0)=0,
  \psi(T)=\psi_\ge(T)=T U_\ge^{y_0} \right\},
\]
where $T>0$ is the number we found in Lemma \ref{lem:1}. To this set
we associate the min-max level
\[
C_\ge = \inf_{\psi \in \Phi_\ge} \max_{s \in [0,T]} \Gamma_\ge
(\psi(s)).
\]
By well-known arguments (see for instance \cite[Proposition 3.2]{CJS}
for a proof in a local setting that extends smoothly to our case) it
is possible to prove that
\begin{displaymath} 
\lim_{\ge \to 0} C_\ge = E_{V_0}.
\end{displaymath}
For $\alpha \in \mathbb{R}$ define the sublevel
\[
\Gamma_\ge^\alpha = \left\{ v \in H \mid \Gamma_\ge (v) \leq \alpha \right\}.
\]
\begin{proposition} \label{prop:4.2}
  Let $d>0$ be small enough, and
  let $\{\ge_j\}_j$ be such that $\lim_{j \to +\infty} \ge_j = 0$ and
  let $\{v_{\ge_j}\} \subset X_{\ge_j}^d$ be such that
\begin{displaymath}
  \lim_{j \to +\infty} \Gamma_{\ge_j}(v_{\ge_j}) \leq E_{V_0}, \quad \lim_{j \to +\infty} \Gamma'_{\ge_j}(v_{\ge_j})=0.
\end{displaymath}
Then there exist --- up to a subsequence --- $\{\tilde y_j \}_j \subset
\mathbb{R}^N$, a point $\bar y \in \mathscr{M}$ and $U \in S_{V_0}$ such
that
\begin{eqnarray*}
&&\lim_{j \to +\infty} |\ge_j \tilde y_j - \bar y | = 0 \\
&& \lim_{j \to +\infty} \left\| v_{\ge_j} - \varphi_{\ge_j}(\cdot,\cdot - \tilde y_j) U(\cdot,\cdot - \tilde y_j) \right\|=0.
\end{eqnarray*}
\end{proposition}
\begin{proof}
  In the proof we will drop the index $j$ and write $\ge$ instead of $\ge_j$ for
  simplicity. By Proposition \ref{prop:comp}, there exist $Z \in
  S_{V_0}$, $\{y_\ge\} \subset \mathscr{M}^\beta$ and $\bar{y} \in
  \mathscr{M}^\beta$ such that $y_\ge \to \bar y$ as $\ge \to 0$ and
\begin{equation} \label{eq:10}
\left\|
v_\ge - \varphi_\ge \left(\cdot,\cdot - \frac{y_\ge}{\ge}\right) Z \left(\cdot,\cdot-\frac{y_\ge}{\ge} \right)
\right\| \leq 2d
\quad\hbox{for every $\ge \ll 1$}.
\end{equation}
We set
\[
v_{1,\ge} = \varphi_\ge \left(\cdot,\cdot - \frac{y_\ge}{\ge}\right) Z \left(\cdot,\cdot-\frac{y_\ge}{\ge} \right),
\quad
v_{2,\ge} = v_\ge - v_{1,\ge}.
\]
We claim that
\begin{equation}\label{eq:8}
  \Gamma_\ge ( v_\ge) \geq \Gamma_\ge (v_{1,\ge}) + \Gamma_\ge (v_{2,\ge}) + O (\ge).
\end{equation}
Suppose that there exist $R>0$ and points
\[
\tilde{y}_\ge \in B\left( \frac{y_\ge}{\ge},\frac{2\beta}{\ge} \right) \setminus
B \left( \frac{y_\ge}{\ge},\frac{\beta}{\ge} \right)
\]
such that
\[
\liminf_{\ge \to 0} \int_{B(\tilde{y}_\ge,R)} \gamma(v_\ge)^2 \, dy >0.
\]
Set $\tilde{v}_\ge(x,y) = v_\ge(x,y + \tilde{y}_\ge)$ so that
\begin{equation} \label{eq:9}
\liminf_{\ge \to 0} \int_{B(0,R)} \gamma(\tilde{v}_\ge)^2 \, dy >0.
\end{equation}
Up to subsequences, we can assume that
\[
\lim_{\ge \to 0} \ge \tilde{y}_\ge = y_0 \in \overline{B(\bar y,2\beta)
  \setminus B(\bar y,\beta)}.
\]
The sequence $\{v_\ge\}$ is bounded in $H$ and hence in every
$L^q(\mathbb{R}^N)$ with $q<2N/(N-1)$. As a consequence,
$\tilde{v}_\ge \to \mathcal{W}$ weakly in $H$ and strongly in
$L_{\mathrm{loc}}^q(\mathbb{R}^N)$ for every $q < 2N/(N-1)$. By
(\ref{eq:9}), $\mathcal{W} \neq 0$. Moreover,
\[
\sqrt{-\Delta + m^2} \mathcal{W} + V(y_0)\mathcal{W} = \left( I_\alpha
  * |\mathcal{W}|^p \right) |\mathcal{W}|^{p-2} \mathcal{W}.
\]
Choosing $R \gg 1$,
\[
\liminf_{\ge \to 0} \int_{(0,+\infty) \times B(\tilde{y}_\ge,R)} \left( |\nabla v_\ge|^2 + m^2 v_\ge^2 \right) \, dx\, dy \geq \frac{1}{2} \int_{\mathbb{R}_{+}^{N+1}} \left( |\nabla \mathcal{W}|^2 + m^2 \mathcal{W}^2 \right) \, dx\, dy.
\]
Since $E_a > E_b$ whenever $a>b$, we have
\[
L_{V(y_0)} (\mathcal{W}) \geq E_{V(y_0)} \geq E_{V_0}.
\]
Hence, for some absolute constant $c_0>0$,
\[
\liminf_{\ge \to 0} \int_{(0,+\infty) \times B(\tilde{y}_\ge,R)} \left( |\nabla v_\ge|^2 + m^2 v_\ge^2 \right) \, dx\, dy \geq
c_0 \cdot L_{V(y_0)}(\mathcal{W}) \geq c_0 \cdot E_{V_0} >0,
\]
and this is a contradiction to the exponential decay at infinity of
$Z$ and the fact that $y_0 \neq \bar{y}$.

Since such a sequence $\{\tilde{y}_\ge\}$ cannot exist, a Lemma of
P.-L. Lions (see \cite[Lemma I.1]{Lions}) implies that
\[
 \limsup_{\ge \to 0} \int_{B(\frac{y_\ge}{\ge},\frac{2\beta}{\ge})\setminus B(\frac{y_\ge}{\ge},\frac{\beta}{\ge})}
 |\gamma(v_\ge)|^{\frac{N+1}{N-1}}\, dy =0.
\]
This, the boundedness of $\{\gamma(v_\ge)\}$ in $L^2$ and the
Hardy-Littlewood-Sobolev inequality imply {\setlength\arraycolsep{2pt}
\begin{eqnarray*}
&&\lim_{\ge \to 0} \left( \int_{\mathbb{R}^N} (I_\alpha * |\gamma(v_\ge)|^{p})|\gamma(v_\ge)|^p \, dy \right.
 \\
&&\qquad {}- \left. \int_{\mathbb{R}^N}
 (I_\alpha * |\gamma(v_{1,\ge})|^{p})|\gamma(v_{1,\ge})|^p \, dy - \int_{\mathbb{R}^N} (I_\alpha * |\gamma(v_{2,\ge})|^{p})|\gamma(v_{2,\ge})|^p \, dy \right)
 =0.
\end{eqnarray*}
}
If we write
{\setlength\arraycolsep{2pt}
\begin{eqnarray*}
&& \Gamma_\ge(v_\ge) = \Gamma_\ge (v_{1,\ge})+\Gamma_\ge(v_{2,\ge}) \\
 &&{}+ \int_{(0,+\infty) \times \left( B(\frac{y_\ge}{\ge},\frac{2\beta}{\ge})\setminus B(\frac{y_\ge}{\ge},\frac{\beta}{\ge}) \right)}
 \varphi_\ge \left( x,y-\frac{y_\ge}{\ge} \right) \left( 1- \varphi_\ge \left( x,y-\frac{y_\ge}{\ge} \right) \right)
 |\nabla v_\ge|^2 \, dx \, dy \\
 &&{}+\int_{B(\frac{y_\ge}{\ge},\frac{2\beta}{\ge})\setminus B(\frac{y_\ge}{\ge},\frac{\beta}{\ge})}
 V_\ge \gamma \left(\varphi_\ge \left( x,y-\frac{y_\ge}{\ge} \right) \right) \left( 1- \gamma \left( \varphi_\ge \left( x,y-\frac{y_\ge}{\ge} \right) \right) \right)
 \gamma(v_\ge)^2 \, dy \\
 &&{}- \frac{1}{2p} \int_{\mathbb{R}^N} (I_\alpha * |\gamma(v_\ge)|^p ) |\gamma(v_\ge)|^p \, dy \\
 &&{}+ \frac{1}{2p} \int_{\mathbb{R}^N}
 (I_\alpha * |\gamma(v_{1,\ge})|^{p})|\gamma(v_{1,\ge})|^p \, dy \\
 &&{}+ \frac{1}{2p} \int_{\mathbb{R}^N} (I_\alpha * |\gamma(v_{2,\ge})|^{p})|\gamma(v_{2,\ge})|^p \, dy + o(1)
\end{eqnarray*}
}
as $\ge \to 0$, we deduce that (\ref{eq:8}) holds true. We now estimate $\Gamma_\ge(v_{2,\ge})$. There results
{\setlength\arraycolsep{2pt}
\begin{eqnarray} \label{eq:11}
\Gamma_\ge(v_{2,\ge}) &\geq& \mathscr{E}_\ge(v_{2,\ge}) \nonumber\\
&=& \frac{1}{2} \int_{\mathbb{R}_{+}^{N+1}} |\nabla v_{2,\ge}|^2 \, dx \, dy + \frac{1}{2} \int_{\mathbb{R}^N} V_\ge \gamma(v_{2,\ge})^2 \, dy \nonumber \\
&&
{}- \frac{1}{2p} \int_{\mathbb{R}^N} (I_\alpha * |\gamma(v_{2,\ge})|^p) |\gamma (v_{2,\ge})|^p \, dy.
\end{eqnarray}
}
For some constant $C>0$ and using again the boundedness of $\{\gamma(v_{2,\ge})\}$ in $L^2$,
\[
 \int_{\mathbb{R}^N} (I_\alpha * |\gamma(v_{2,\ge})|^p) |\gamma (v_{2,\ge})|^p \, dy \leq C \|v_{2,\ge}\|.
\]
Now (\ref{eq:10}) implies that $\|v_{2,\ge}\| \leq 4d$ for small
values of $\ge$. Taking $d>0$ sufficiently small uniformly with
respect to $\ge$, we have
\begin{displaymath} 
 \frac{1}{2} \|v_{2,\ge}\|^2 - \frac{1}{2p} \int_{\mathbb{R}^N} (I_\alpha * |\gamma(v_{2,\ge})|^p) |\gamma (v_{2,\ge})|^p \, dy
 \geq \frac{1}{8} \|v_{2,\ge}\|^2.
\end{displaymath}
Since the functional $\mathscr{E}_\ge$ is uniformly bounded in $X_\ge^d$ for small $\ge>0$, the penalization term $Q_\ge$ is uniformly bounded in $X_\ge^d$ for small $\ge>0$ as well.
As a consequence, for an absolute constant $C>0$,
\begin{equation} \label{eq:13}
 \int_{\mathbb{R}^N \setminus O_\ge} \gamma(v_{2,\ge})^2 \, dy \leq C \ge^{\frac{6}{\mu}},
\end{equation}
and (\ref{eq:11}--\ref{eq:13}) imply $\Gamma(v_{2,\ge}) \geq o(1)$ as $\ge \to 0$.

Let us introduce
\[
 v_{1,\ge}^1(x,y) =
 \left\{
 \begin{array}{ll}
    v_{1,\ge}(x,y) &\hbox{if $y \in O_\ge$} \\
    0 &\hbox{otherwise.}
 \end{array}
 \right.
\]
For $\mathfrak{W}_\ge(x,y)=v_{1,\ge}^1(x,y+y_\ge/\ge)$, we can proceed as before and conclude that $\mathfrak{W}_\ge$
converges weakly in $L^q(\mathbb{R}_{+}^{N+1})$, $q<2N/(N-1)$, to a solution $\mathfrak{W}$ of
\[
 \sqrt{-\Delta +m^2} \mathfrak{W} + V(\bar{y})\mathfrak{W} = (I_\alpha * |\mathfrak{W}|^{p})|\mathfrak{W}|^{p-2}\mathfrak{W}.
\]
We claim that $\mathfrak{W}_\ge$ converges to $\mathfrak{W}$ strongly
in $H$. As before, assume the existence of a radius $R>0$ and of a
sequence $\{z_\ge\}\subset \mathbb{R}^N$ such that $z_\ge \in
B(y_\ge/\ge,2\beta/\ge)$,
\[
\liminf_{\ge \to 0} \left| z_\ge - \ge^{-1}y_\ge \right| =0
\quad\hbox{and}\quad \liminf_{\ge \to 0} \int_{B(z_\ge,R)} |\gamma
(v_{1,\ge}^1)|^2 \, dy >0.
\]
Without loss of generality, $\ge z_\ge \to z \in O$ as $\ge \to
0$. Then $\widetilde{\mathfrak{W}}_\ge(x,y) = \mathfrak{W}_\ge
(x,y+z_\ge)$ converges weakly in $L^q(\mathbb{R}_{+}^{N+1})$,
$q<2N/(N-1)$, to some $\widetilde{\mathfrak{W}} \in H$ that solves
\[
\sqrt{-\Delta +m^2} \widetilde{\mathfrak{W}} +
V(z)\widetilde{\mathfrak{W}} = (I_\alpha *
|\widetilde{\mathfrak{W}}|^{p})|\widetilde{\mathfrak{W}}|^{p-2}\widetilde{\mathfrak{W}}.
\]
and we obtain a contradiction as before. Again,
\begin{equation} \label{eq:17} \lim_{\ge \to 0} \int_{\mathbb{R}^N}
  (I_\alpha *
  |\gamma(\mathfrak{W}_\ge)|^{p})|\gamma(\mathfrak{W}_\ge)|^p \, dy =
  \int_{\mathbb{R}^N} (I_\alpha * |\gamma(\mathfrak{W})|^{p})
  |\gamma(\mathfrak{W})|^p\, dy.
\end{equation}
Hence
{\setlength\arraycolsep{2pt}
\begin{eqnarray*} 
\limsup_{\ge \to 0} \Gamma_\ge(v_{1,\ge}^1) &\geq& \liminf_{\ge \to 0} \Gamma_\ge(v_{1,\ge}^1) \nonumber\\
&\geq& \liminf_{\ge \to 0} \frac{1}{2} \int_{(0,+\infty) \times B(0,R)} |\nabla \mathfrak{W}_\ge|^2 \, dx\, dy \nonumber \\
&&{}+\frac{1}{2}\int_{B(0,R)} V(\ge y+y_\ge) |\gamma(\mathfrak{W}_\ge)|^2\, dy \nonumber \\
&&{}-\frac{1}{2p}\int_{\mathbb{R}^N} (I_\alpha * |\gamma(\mathfrak{W}_\ge)|^p)|\gamma(\mathfrak{W}_\ge)|^p \, dy \nonumber \\
&\geq& \frac{1}{2} \int_{(0,+\infty) \times B(0,R)} |\nabla \mathfrak{W}_\ge|^2 \, dx\, dy \nonumber \\
&&{}+ \frac{1}{2}V(\bar{y}) \int_{B(0,R)}  |\gamma(\mathfrak{W}_\ge)|^2\, dy \nonumber \\
&&{}-\frac{1}{2p} \int_{\mathbb{R}^N} (I_\alpha * |\gamma(\mathfrak{W})|^{p}) |\gamma(\mathfrak{W})|^p\, dy.
\end{eqnarray*}
}
Since $R>0$ is arbitrary,
{\setlength\arraycolsep{2pt}
\begin{eqnarray} \label{eq:15}
\limsup_{\ge \to 0} \Gamma_\ge(v_{1,\ge}^1) &\geq& \frac{1}{2} \int_{\mathbb{R}_{+}^{N+1}} |\nabla \mathfrak{W}|^2\, dx\, dy +
\frac{1}{2} V(\bar{y}) \int_{\mathbb{R}^N} |\gamma(\mathfrak{W})|^2\, dy \nonumber \\
&&{}-\int_{\mathbb{R}^N} (I_\alpha *|\gamma(\mathfrak{W})|^p) |\gamma(\mathfrak{W})|^p \, dy \nonumber \\
&=& L_{V(\bar{y})} (\mathfrak{W}) \nonumber \\
&\geq& E_{V_0}.
\end{eqnarray}
}
Recalling (\ref{eq:8}), we find
\[
 \limsup_{\ge \to 0} \left(
 \Gamma_\ge (v_{2,\ge})+\Gamma_\ge(v_{1,\ge}^1)
 \right) = \limsup_{\ge \to 0} \left(
 \Gamma_\ge (v_{2,\ge})+\Gamma_\ge(v_{1,\ge})
 \right) \leq \limsup_{\ge \to 0} \Gamma_\ge (v_\ge) \leq E_{V_0}.
\]
Now $\Gamma_\ge(u_{2,\ge}) \geq o(1)$ yields
\[
 \lim_{\ge \to 0} \Gamma_\ge(v_{1,\ge}^1) = E_{V_0}.
\]
What we have just proved entails that $L_{V(\bar{y})}(\mathfrak{W}) = E_{V_0}$,
and then $\bar{y} \in \mathscr{M}$.  As a consequence, $\mathfrak{W}$ is, up to a translation in the $y$-variable,
an element of $S_{V_0}$, namely $\mathfrak{W}(x,y)=U(x,y-\mathfrak{z})$ for some $U \in S_{V_0}$
and some $\mathfrak{z}\in\mathbb{R}^N$.

Recalling that $V \geq V(\bar{y})$ on the subset $O$ and using the
identity $L_{V(\bar{y})}(\mathfrak{W}) = E_{V_0}$ we get
{\setlength\arraycolsep{2pt}
\begin{eqnarray} \label{eq:16}
 &&\int_{\mathbb{R}_{+}^{N+1}} |\nabla \mathfrak{W}|^2 \, dx\, dy + V_0\int_{\mathbb{R}^N} |\gamma
 (\mathfrak{W})|^2\, dy - 2p \int_{\mathbb{R}^N} (I_\alpha * |\gamma(\mathfrak{W})|^p) |\gamma(\mathfrak{W})|^p \, dy \nonumber \\
 &&\quad \geq  \limsup_{\ge \to 0} \int_{\mathbb{R}_{+}^{N+1}} |\nabla \mathfrak{W}_\ge|^2 \, dx\, dy
 + \int_{\mathbb{R}^N} V(\ge y+y_\ge) |\gamma(\mathfrak{W}_\ge)|^2 \, dy \nonumber \\
 &&\quad\quad{}-2p \int_{\mathbb{R}^N} (I_\alpha * |\gamma(\mathfrak{W}_\ge)|^p) |\gamma(\mathfrak{W}_\ge)|^p \, dy \nonumber \\
 &&\quad\geq \limsup_{\ge \to 0} \int_{\mathbb{R}_{+}^{N+1}} |\nabla \mathfrak{W}_\ge|^2 \, dx\, dy
 + \int_{\mathbb{R}^N} V(\bar{y}) |\gamma(\mathfrak{W}_\ge)|^2 \, dy \nonumber \\
 &&\quad\quad{}-2p \int_{\mathbb{R}^N} (I_\alpha * |\gamma(\mathfrak{W}_\ge)|^p) |\gamma(\mathfrak{W}_\ge)|^p \, dy \nonumber \\
 &&\quad\geq \int_{\mathbb{R}_{+}^{N+1}} |\nabla \mathfrak{W}|^2 \, dx\, dy + V_0\int_{\mathbb{R}^N} |\gamma
 (\mathfrak{W})|^2\, dy - 2p \int_{\mathbb{R}^N} (I_\alpha * |\gamma(\mathfrak{W})|^p) |\gamma(\mathfrak{W})|^p \, dy,
\end{eqnarray}
}
and therefore
\[
 \lim_{\ge \to 0} \int_{\mathbb{R}^N} V(\ge y+y_\ge)|\gamma(\mathfrak{W}_\ge)|^2 \, dy = \int_{\mathbb{R}^N} V(\bar{y})|\gamma(\mathfrak{W})|^2 \, dy.
\]
Using again the fact that $V \geq V(\bar{y})$ on the subset $O$ we
conclude that $\gamma(\mathfrak{W}_\ge) \to \gamma(\mathfrak{W})$
strongly in $L^2(\mathbb{R}^N)$. Finally, from (\ref{eq:17}),
(\ref{eq:15}) and (\ref{eq:16}) we see that
{\setlength\arraycolsep{2pt}
\begin{multline*}
\int_{\mathbb{R}_{+}^{N+1}} |\nabla \mathfrak{W}|^2 \, dx\, dy +
\int_{\mathbb{R}^N} V(\bar{y})|\gamma (\mathfrak{W})|^2 \, dy \\
\geq
\limsup_{\ge \to 0} \int_{\mathbb{R}_{+}^{N+1}} |\nabla
\mathfrak{W}_\ge|^2 \, dx\, dy + \int_{\mathbb{R}^N} V(\bar{y})|\gamma
(\mathfrak{W}_\ge)|^2 \, dy.
\end{multline*}
}
The strong convergence of $\mathfrak{W}_\ge$ to $\mathfrak{W}$ in $H$
is now proved. Thus
\[
 v_{1,\ge}^1 (x,y) = U\left(x,y-\frac{y_\ge}{\ge}-\mathfrak{z} \right) + o(1),
\]
and straightforward algebraic manipulations show that
\[
 v_{1,\ge}(x,y) = \varphi_\ge \left( x,y-\frac{y_\ge}{\ge}-\mathfrak{z} \right) U \left( x,y-\frac{y_\ge}{\ge}-\mathfrak{z} \right) +o(1)
\]
strongly in $H$. But $E_{V_0} \geq \lim_{\ge \to 0}
\Gamma_\ge(v_\ge)$ and $\lim_{\ge \to 0} \Gamma_\ge(v_{1,\ge}) =
E_{V_0}$, so that $\lim_{\ge \to 0} \Gamma_\ge (v_{2,\ge})=0$ by
(\ref{eq:8}). Using (\ref{eq:11}) and (\ref{eq:13}) we discover that
$v_{2,\ge} \to 0$ strongly in $H$. This completes the proof.

\section{Critical points of the penalized functional}

We are now ready to show that the penalized functional $\Gamma_\ge$
possesses a critical point for every $\ge >0$ sufficiently small.
\begin{lemma} \label{lem:5.1}
For $d>0$ sufficiently small, there exist positive constants
$\ge_0$ and $\omega$ such that $|\Gamma'_\ge(v)| \geq \omega$
for every $v \in \Gamma_\ge^{D_\ge} \cap \left( X_\ge^d \setminus X_\ge^{d/2}
\right)$ and $\ge \in \left( 0,\ge_0 \right)$.
\end{lemma}
\begin{proof}
  If not, for $d>0$ so small that Proposition \ref{prop:4.2} applies,
  there exist sequences $\{\ge_j\}_j$ with $\lim_j \ge_j =0$ and
  $\{v_{\ge_j}\}_j$ with $v_{\ge_j} \in X_{\ge_j}^d \setminus
  X_{\ge_j}^{d/2}$ satisfying
\begin{displaymath}
  \lim_{j \to+\infty}
  \Gamma_{\ge_j}(v_{\ge_j}) \leq E_{V_0} \quad\hbox{and}\quad
  \lim_{j \to
    +\infty} \Gamma'_{\ge_j}(v_{\ge_j})=0.
\end{displaymath}
Hence Proposition \ref{prop:4.2} applies and provides points
$y_{\ge_j} \in \mathbb{R}^N$, $\bar y \in \mathscr{M}$ and a ground state
$U \in S_{V_0}$ such that
\begin{eqnarray}
&&\lim_{j \to +\infty} |\ge_j y_j - \bar y | = 0 \nonumber \\
&& \lim_{j \to +\infty} \left\| v_{\ge_j} - \varphi_{\ge_j}(\cdot,\cdot - y_j) U(\cdot,\cdot - y_j) \right\|=0. \label{eq:25}
\end{eqnarray}
The definition of $X_{\ge_j}$ implies $\lim_{j \to +\infty}
\mathrm{dist}\, (v_{\ge_j},X_{\ge_j})=0$, and this contradicts the
assumption $v_{\ge_j} \notin X_{\ge_j}^{d/2}$.
\end{proof}
Let now $d>0$ be chosen so that Lemma \ref{lem:5.1} applies.
\begin{proposition} \label{prop:5.2}
For $\ge>0$ sufficiently small, the functional $\Gamma_\ge$
has a critical point $v_\ge \in X_\ge^d \cap \Gamma_\ge^D$.
\end{proposition}
\begin{proof}
  Pick $R_0>0$ so large that $O \subset \left(\{0\} \times
    \mathbb{R}^N \right) \cap B(0,R_0)$ and $\psi_\ge(s) \in H_0^1
  (B(0,R/\ge))$ for any $s \in [0,T]$, $R>R_0$ and $\ge>0$
  sufficiently small. We write $D_\ge = \max_{0 \leq s \leq T}
  \Gamma_\ge(\psi_\ge(s))$.  By Lemma \ref{lem:1}, there exists
  $\mathfrak{a}\in (0,E_{V_0})$ such that, for sufficiently small
  $\ge>0$,
\[
\Gamma_\ge(\psi_\ge(s)) \geq D_\ge -\mathfrak{a} \quad\hbox{implies}\quad
\psi_\ge(s) \in X_\ge^{d/2}\cap H_0^1 (B(0,R/\ge)).
\]
We claim that, for sufficiently small $\ge>0$ and $R>R_0$, there is a
sequence $\{v_n^R\}_n \subset X_\ge^{d/2} \cap \Gamma_\ge^{D_\ge} \cap
H_0^1(B(0,R/\ge))$ such that $\Gamma'_\ge(v_n^R) \to 0$ is
$H_0^1(B(0,R/\ge))$ as $n \to +\infty$.

Arguing by contradiction, we assume that for sufficiently small
$\ge>0$ there exists a number $a_R(\ge)>0$ such that
\[
|\Gamma'_\ge(v)|
\geq a_R(\ge)
\]
on $X_\ge^{d/2} \cap \Gamma_\ge^{D_\ge} \cap
H_0^1(B(0,R/\ge))$. With a slight abuse of notation, we will identify
any $v \in H_0^1(B(0,R/\ge))$ with its extension to $H$ as the null
function outside $B(0,R/\ge)$. Applying Lemma \ref{lem:5.1}, we find a
number $\omega>0$, independent of $\ge>0$, such that $|\Gamma'_\ge(v)|
\geq \omega$ for $v \in \Gamma_\ge^{D_\ge} \cap (X_\ge^d \setminus
X_\ge^{d/2})$. By a classical deformation argument that starts from
$\psi_\ge$, there exist some $\mu \in (0,\mathfrak{a})$ and a path $\psi \in
C([0,T],H)$ satisfying
\[
\psi(s)=\psi_\ge(s) \ \hbox{for $\psi_\ge(s) \in \Gamma_\ge^{D_\ge-\mathfrak{a}}$},
\quad
\psi(s) \in X_\ge^d \ \hbox{for $\psi_\ge(s) \notin \Gamma_\ge^{D_\ge-\mathfrak{a}}$}
\]
and
\begin{equation} \label{eq:21}
\Gamma_\ge(\psi(s)) < D_\ge-\mu \quad\hbox{for every $s \in [0,T]$}.
\end{equation}
Let $\zeta \in C_0^\infty(\mathbb{R}_{+}^{N+1})$ be a cut-off function
such that $\zeta(x,y)=1$ for $0<x<\delta$ and $y \in O^\delta$, $\zeta
(x,y)=0$ for $x \geq 2\delta$ and $y \notin O^{2\delta}$, $\zeta
(\cdot,\cdot) \in [0,1]$, and $|\nabla \zeta| \leq 2/\delta$. For
$\psi(s) \in X_\ge^d$ we denote $\psi_1(s)=\zeta_\ge \psi(s)$ and
$\psi_2(s)=(1-\zeta_\ge)\psi(s)$, where $\zeta_\ge (x,y)=\zeta (\ge x,
\ge y)$. We remark that we understand the dependency on $\ge$ in the
notation of $\psi_1$ and $\psi_2$. Observe that
{\setlength\arraycolsep{2pt}
\begin{eqnarray*}
 \Gamma_\ge(\psi(s)) &=& \Gamma_\ge (\psi_1(s))+\Gamma_\ge(\psi_2(s))
 + Q_\ge(\psi(s)) - Q_\ge(\psi_1(s)) - Q_\ge (\psi_2(s)) \\
 &&{}- \frac{1}{2p} \int_{\mathbb{R}^N}  (I_\alpha * |\gamma(\psi(s))|^p) |\gamma(\psi(s))|^p
 \\
 &&{}+ \frac{1}{2p} \int_{\mathbb{R}^N} (I_\alpha * |\gamma(\psi_1(s))|^p) |\gamma(\psi_1(s))|^p \\
 &&{}+ \frac{1}{2p} \int_{\mathbb{R}^N} (I_\alpha * |\gamma(\psi_2(s))|^p) |\gamma(\psi_2(s))|^p.
\end{eqnarray*}
}
The elementary inequality $(h+k-1)_{+} \geq (h-1)_{+} + (k-1)_{+}$ valid for $h \geq 0$ and $k \geq 0$
immediately implies that
\[
 Q_\ge(\psi(s)) \geq Q_\ge(\psi_1(s)) + Q_\ge (\psi_2(s))
\]
and, similarly to (\ref{eq:13}), we find that
\begin{equation} \label{eq:18}
 \int_{\mathbb{R}^N \setminus O_\ge} |\gamma(\psi(s))|^2\, dy \leq C \ge^{6/\mu}.
\end{equation}
On the other hand, writing $\kappa = (I_\alpha * |\gamma(\psi(s))|^p) |\gamma(\psi(s))|^p - (I_\alpha * |\gamma(\psi_1(s))|^p) |\gamma(\psi_1(s))|^p
-(I_\alpha * |\gamma(\psi_2(s))|^p) |\gamma(\psi_2(s))|^p$ for simplicity,
{\setlength\arraycolsep{2pt}
\begin{eqnarray*}
 \int_{\mathbb{R}^N} \kappa &=& 2 \int_{O_\ge^{2\delta} \times (\mathbb{R}^N \setminus O_\ge^{2\delta})} (I_\alpha * |\gamma(\psi(s))|^p)|\gamma(\psi(s))|^p \\
&&{} -2 \int_{(O_\ge^{2\delta} \setminus O_\ge^\delta) \times (\mathbb{R}^N \setminus O^\delta)} (I_\alpha * |\gamma(\psi(s))|^p)|\gamma(\psi(s))|^p
\end{eqnarray*}
}
and from (\ref{eq:18}) via interpolation we deduce that
{\setlength\arraycolsep{2pt}
\begin{eqnarray}
 \lim_{\ge \to 0} \int_{O_\ge^{2\delta} \times (\mathbb{R}^N \setminus O_\ge^{2\delta})} (I_\alpha * |\gamma(\psi(s))|^p)|\gamma(\psi(s))|^p &=& 0 \label{eq:19} \\
 \lim_{\ge \to 0} \int_{(O_\ge^{2\delta} \setminus O_\ge^\delta) \times (\mathbb{R}^N \setminus O_\ge^\delta)} (I_\alpha * |\gamma(\psi(s))|^p)|\gamma(\psi(s))|^p &=& 0 \label{eq:20}.
\end{eqnarray}
}
Equations (\ref{eq:19}) and (\ref{eq:20}) yield
{\setlength\arraycolsep{2pt}
\begin{eqnarray*}
\lim_{\ge \to 0}
&&\int_{\mathbb{R}^N}  (I_\alpha * |\gamma(\psi(s))|^p) |\gamma(\psi(s))|^p
 -  \int_{\mathbb{R}^N} (I_\alpha * |\gamma(\psi_1(s))|^p) |\gamma(\psi_1(s))|^p \\
 &&\quad{}-  \int_{\mathbb{R}^N} (I_\alpha * |\gamma(\psi_2(s))|^p) |\gamma(\psi_2(s))|^p =0,
\end{eqnarray*}
}
and hence, as $\ge \to 0$,
\[
 \Gamma_\ge(\psi(s)) \geq \Gamma_\ge(\psi_1(s)) + \Gamma_\ge(\psi_2(s)) + o(1).
\]
By similar arguments,
\begin{multline*}
 \Gamma_\ge(\psi_2(s))  \\
 \geq -\frac{1}{2p} \int_{(\mathbb{R}^N \setminus O_\ge) \times (\mathbb{R}^N \setminus O_\ge)}
 I_\alpha (x-y)
 |\gamma(\psi_2(s)(x)|^p |\gamma(\psi_2(s)(y)|^p\, dx\, dy \geq o(1),
\end{multline*}
and we finally conclude that
\begin{displaymath}
 \Gamma_\ge(\psi(s)) \geq \Gamma_\ge(\psi_1(s))+o(1)
\end{displaymath}
as $\ge \to 0$. If we define
\[
 \psi_1^1(s)(x,y) = \left\{
 \begin{array}{ll}
  \psi_1(s)(x,y) &\hbox{if $x>0$ and $y \in O^{2\delta}$} \\
  0 &\hbox{if $x>0$ and $y \notin O^{2\delta}$},
 \end{array}
\right.
\]
we immediately see that $\Gamma_\ge(\psi_1(s)) \geq
\Gamma_\ge(\psi_1^1(s))$, and $\psi_1^1 \in \Phi_\ge$ because
$0<\mathfrak{a} < E_{V_0}$. Now \cite[Proposition 3.4]{CZR} implies
that, as $\ge \to 0$,
\[
 \max_{0 \leq s \leq T} \Gamma_\ge (\psi(s)) \geq E_{V_0} + o(1),
\]
and this contradicts (\ref{eq:21}).

For a fixed $\ge$ sufficiently small and for $R \gg1$, we consider a
sequence $\{v_n^R\}_n \subset X_\ge^{d/2} \cap \Gamma_\ge^{D_\ge} \cap
H_0^1(B(0,R/\ge))$ such that $\Gamma'_\ge(v_n^R) \to 0$ is
$H_0^1(B(0,R/\ge))$ as $n \to +\infty$. The boundedness of
$\{v_n^R\}_n$ in $H_0^1(B(0,R/\ge))$ and the Sobolev embedding theorem
imply that $v_n^R \to v^R$ strongly in $L^q(B(0,R/\ge))$ for any
$q<2N/(N-1)$. Since $\{v_n^R\}_n$ is a Palais-Smale sequence, a
standard argument shows that $v_n^R \to v^R$ strongly in
$H_0^1(B(0,R/\ge))$. Hence the limit $v^R$ is a weak solution to the
problem
\begin{displaymath}
  -\Delta v^R + m^2 v^R =0 \quad\hbox{in $B \left( 0,\frac{R}{\ge} \right)$}
\end{displaymath}
with
{\setlength\arraycolsep{2pt}
\begin{eqnarray*}
    -\frac{\partial v^R}{\partial x}(0,y) &=& -V_\ge(y)v^R(0,y) +
      \left( I_\alpha * |v^R(0,\cdot)|^{p} \right) |v^R(0,y)|^{p-2}v^R(0,y) + \nonumber \\
      &&\quad{}+(2p+1) \left( \int_{\mathbb{R}^N} \chi_\ge \gamma(v^R)^2 \, dy -1 \right)_{+}^{\frac{2p-1}{2}}
      \chi_\ge v^R(0,y)
\end{eqnarray*}
}
for $y \in \mathbb{R}^N$ with $|y| = R/\ge$.

Since $v^R \in X_\ge^{d} \cap \Gamma_\ge^{D_\ge} \cap
H_0^1(B(0,R/\ge))$, we deduce that both $\{\|v^R\|\}_R$ and
$\{\Gamma_\varepsilon(v^R)\}_R$ are uniformly bounded for
$\varepsilon>0$ sufficiently small. Hence also
$\{Q_\varepsilon(v^R)\}_R$ is uniformly bounded for $\varepsilon>0$
sufficiently small.  Now a Moser iteration scheme like \cite[Theorem
3.2]{CZN1} yields that $\{v^R\}_R$ is bounded in $L^\infty$ uniformly
for $\ge>0$ sufficiently small. Taking into account that
$\{Q_\varepsilon(v^R)\}_R$ is uniformly bounded in $L^\infty$ and
\[
 \left( I_\alpha * |v^R(0,\cdot)|^{p} \right) |v^R(0,y)|^{p-1} \leq \frac{1}{2}(V_\ge + m)  |v^R(0,y)|
\]
when $|y| \geq 2R/\ge$, we can perform a comparison argument as in \cite[Theorem 5.1]{CZN1} and derive 
\begin{displaymath}
|v^R(x,y)| \leq C e^{-m (\sqrt{x^2+|y|^2}-2R_0)}.
\end{displaymath}
We assume, without loss of generality, that $\{v^R\}_R$ weakly converges to some $v_\ge$ in $H$ as $R \to +\infty$
that solves
\begin{equation} \label{eq:22}
  -\Delta v_\ge + m^2 v_\ge =0 \quad\hbox{in $\mathbb{R}_{+}^{N+1}$}
\end{equation}
with
{\setlength\arraycolsep{2pt}
\begin{eqnarray} \label{eq:23}
    -\frac{\partial v_\ge}{\partial x}(0,y) &=& -V_\ge(y)v_\ge(0,y) +
      \left( I_\alpha * |v_\ge(0,\cdot)|^{p} \right) |v_\ge(0,y)|^{p-2}v_\ge(0,y) + \nonumber \\
      &&\quad{}+(2p+1) \left( \int_{\mathbb{R}^N} \chi_\ge \gamma(v_\ge)^2 \, dy -1 \right)_{+}^{\frac{2p-1}{2}}
      \chi_\ge v_\ge(0,y)
\end{eqnarray}
}
for $y \in \mathbb{R}^N$.
\end{proof}
\section{Proof of the Theorem \ref{th:main}}

We can now collect all the results of the previous section to prove
our main existence theorem. To begin with, Proposition \ref{prop:5.2}
gives us a number $\ge_0>0$ such that, for $0<\ge<\ge_0$, the
penalized functional $\Gamma_\ge$ possesses a critical point $v_\ge
\in X_\ge^d \cap \Gamma_\ge^{D_\ge}$. As in the proof of
Proposition~\ref{prop:comp}, we have $v_\ge \in \bigcap_{q>2}
L^q(\mathbb{R}_{+}^{N+1})$, and $\{v_\ge\}$ is bounded
$L^\infty([0,+\infty) \times \mathbb{R}^N)$. By the results of
Proposition \ref{prop:4.2},
\[
 \lim_{\ge \to 0} \int_{\mathbb{R}_{+}^{N+1} \setminus \left( [0,+\infty) \times (\mathscr{M}^{2\beta})_\ge\right)}
 \left( |\nabla v_\ge|^2 + V_\ge |v_\ge|^2 \right) \, dx\, dy =0.
\]
It now follows that
\[
 \lim_{\ge \to 0} \sup_{(x,y) \in \mathbb{R}_{+}^{N+1} \setminus \left( [0,+\infty) \times (\mathscr{M}^{2\beta})_\ge\right)}
 |v_\ge(x,y)| =0,
\]
and as in the last step of the previous section we deduce an
exponential decay of the trace~$u_\ge$ away from $\mathbb{R}^N \setminus (\mathscr{M}^{2\beta})_\ge$:
\[
 |u_\ge(y)| \leq C_1 \exp \left(
 -C_2 \mathop{\mathrm{dist}} \left( y, (\mathscr{M}^{2\beta})_\ge \right)
 \right).
\]
Taking $\ge$ smaller, this estimate implies that $Q_\ge(v_\ge)=0$, and (\ref{eq:22})-(\ref{eq:23}) are
the local Neumann problem in the half space $\R^N$ corresponding to the nonlocal problem
(\ref{eq:1bis}). The conclusion now follows by reversing the local realization of the operator $\sqrt{-\Delta +m^2}$. Recalling (\ref{eq:25}) and all the scalings, we immediately deduce (\ref{eq:26}). This completes the proof.
\end{proof}

\section*{Acknowledgements}
The first author is partially supported by GNAMPA-INDAM Project 2014
\emph{Aspetti differenziali e geometrici nello studio di problemi
  ellittici quasilineari.}  The second author is partially supported
by the FIRB 2012 project \emph{Dispersive equations and Fourier
  analysis} and by the PRIN 2012 project \emph{Critical point theory
  and perturbative methods for nonlinear differential equations}.
  
The authors wish to express their gratitude to the anonymous referee for their important remarks.

\end{document}